\documentclass[12pt,a4paper]{article}
\usepackage{latexsym}
\usepackage{amssymb}
\usepackage{amsmath}
\usepackage{amsthm}
\usepackage{thmtools}
\usepackage{enumitem}
\usepackage{cite}
\usepackage[all]{xy}
\usepackage{framed}
\usepackage{hyperref}
\usepackage{graphicx}
\usepackage{parskip}
\usepackage{hyperref}
\usepackage{bm}
\usepackage[a4paper,margin=2cm]{geometry}
\usepackage{ipa}
\usepackage[utf8x]{inputenc}
\usepackage{mathdots}
\usepackage{pifont}
\usepackage{marginnote}
\usepackage{xcolor}
\usepackage{marginnote}
\usepackage[normalem]{ulem}
\newcounter{tmp}

\begingroup
    \makeatletter
    \@for\theoremstyle:=definition,remark,plain\do{%
        \expandafter\g@addto@macro\csname th@\theoremstyle\endcsname{%
            \addtolength\thm@preskip\parskip
            }%
        }
\endgroup

\makeatletter
\xydef@\txt@ii#1{\vbox{\vspace*{-5pt}%
 \let\\=\cr
 \tabskip=\z@skip \halign{\relax\hfil\txtline@@{##}\hfil\cr\leavevmode#1\crcr}}}
\makeatother

\theoremstyle{definition}
\newtheorem{thm}{Theorem}[section]
\newtheorem{lem}[thm]{Lemma}
\newtheorem{cor}[thm]{Corollary}
\newtheorem{defn}[thm]{Definition}
\newtheorem{propn}[thm]{Proposition}
\newtheorem*{thm*}{Theorem}

\newtheorem*{qn*}{Question}

\newtheorem{props}[thm]{Properties}

\theoremstyle{remark}
\newtheorem{rk}[thm]{Remark}

\newtheorem{ex}[thm]{Example}

\newtheoremstyle{custthm}{\parskip}{}{\normalfont}{}{\bfseries}{.}{ }{\thmname{#1} \thmnote{#3}}
\theoremstyle{custthm}

\newcommand{\inn}{\mathrm{inn}}
\newcommand{\Aut}{\mathrm{Aut}}

\newcommand{\udim}{\mathrm{udim}}
\newcommand{\rudim}{\mathrm{r.udim}}
\renewcommand{\O}{\mathcal{O}}
\renewcommand{\c}{\mathbf{c}}
\renewcommand{\d}{\mathbf{d}}

\begin{document}

\numberwithin{equation}{section}
\binoppenalty=\maxdimen
\relpenalty=\maxdimen

\title{Filtered skew derivations on simple artinian rings}
\author{Adam Jones, William Woods}
\date{\today}
\maketitle
\begin{abstract}
\noindent Given a complete, positively filtered ring $(R,f)$ and a compatible skew derivation $(\sigma,\delta)$, we may construct its skew power series ring $R[[x;\sigma,\delta]]$. Due to topological obstructions, even if $\delta$ is an \emph{inner} $\sigma$-derivation, in general we cannot ``untwist" it, i.e. reparametrise to find a filtered isomorphism $R[[x; \sigma, \delta]] \cong R[[x'; \sigma]]$, as might be expected from the theory of skew polynomial rings; similarly when $\sigma$ is an inner automorphism. We find general conditions under which it is possible to untwist the multiplication data, and use this to analyse the structure of $R[[x;\sigma,\delta]]$ in the simplest case when $R$ is a matrix ring over a (noncommutative) noetherian discrete valuation ring.
%\blfootnote{\emph{2010 Mathematics Subject Classification}: x.}
%
\end{abstract}
%\vspace{-16pt}

%\newpage
\tableofcontents

\newpage
\section{Introduction}

Let $R$ be a ring and $(\sigma,\delta)$ a skew derivation on $R$: that is, $\sigma$ is an automorphism of $R$, and $\delta$ is a left $\sigma$-derivation on $R$, i.e. a linear map $R\to R$ satisfying $\delta(ab) = \delta(a)b + \sigma(a)\delta(b)$ for all $a,b\in R$. Then we may define the \emph{skew polynomial ring} $R[x;\sigma,\delta]$ as the unique ring which is equal to $R[x]$ as a left $R$-module and whose multiplication is given by $xr = \sigma(r)x + \delta(r)$ for all $r\in R$.

In the theory of skew polynomial rings (see \S \ref{subsec: reparametrising skew polynomials} below), it is well known that, if $\sigma$ is an \emph{inner} automorphism of $R$, then there exists some $x'\in R[x;\sigma,\delta]$ such that $R[x;\sigma,\delta] = R[x';\delta']$ for some derivation $\delta'$; and likewise, if $\delta$ is an \emph{inner} $\sigma$-derivation of $R$, then there exists some $x'\in R[x;\sigma,\delta]$ such that $R[x;\sigma,\delta] = R[x';\sigma]$. This is a crucial and frequently used simplification in the theory: see e.g. \cite[\S 2.3(iii), Theorem 14.9]{GooLet94} or \cite[Proposition 3.10]{Goo92}.

We are primarily interested in skew \emph{power series} rings, where there are many extra topological difficulties to deal with.

Firstly: given an arbitrary ring $R$ and an arbitrary skew derivation $(\sigma, \delta)$ on $R$, it is in general not true that there exists a well-defined multiplication of the above form on the left module $R[[x]]$ without imposing some kind of convergence condition on the multiplication data. Our primary motivation comes from studying the completed group algebras of certain finite-rank pro-$p$ groups (these completed group algebras are also known as \emph{Iwasawa algebras}), where the following notions are appropriate. If $(R,v)$ is a complete, $\mathbb{N}$-filtered ring, and $(\sigma,\delta)$ is compatible with $v$ (in the sense of Definition \ref{defn: compatible} below), then we can define the \emph{skew power series ring}
$$R[[x; \sigma, \delta]] = \left\{ \sum_{n\geq 0} r_n x^n : r_n\in R\right\},$$
which is also a complete $\mathbb{N}$-filtered ring. (If $v$ has values in $\mathbb{Z}\cup\{\infty\}$ rather than $\mathbb{N}\cup\{\infty\}$, then we can define an appropriate notion of \emph{bounded} skew power series ring -- see \cite{jones-woods-1} for details -- but we do not deal with such rings in this paper.)

Secondly: suppose that there exists $t\in R$ such that $\delta(r) = tr - \sigma(r)t$, an \emph{inner} $\sigma$-derivation. Then, if we reparametrise the skew polynomial ring by changing our variable from $x$ to $x' = x + t$, we find that $R[x;\sigma,\delta] = R[x';\sigma]$: in passing from $x$ to $x'$, we will say that we have \emph{untwisted} $\delta$ from $R[x;\sigma,\delta]$. This is beneficial as rings of the form $R[x';\sigma]$ (with zero derivation) are typically much easier to understand. (There is a similar procedure by which inner automorphisms $\sigma$ can be \emph{untwisted}.) However, under a reparametrisation like this, it is generally not true that we will have $R[[x]] = R[[x']]$ even as left $R$-modules (see Example \ref{ex: when reparametrisation doesn't work}), meaning that this simplification is not always available.

\subsection{Maximal orders in semisimple artinian rings}

The main result of this paper uses this notion of untwisting to analyse the structure of skew power series rings. To state this result, we need to impose further conditions on the base ring, since general filtered rings are too pathological for us to be able to say much of consequence.

The rings of interest, which we denote by $\O$, will typically be specific maximal orders in certain complete, filtered semisimple artinian rings $Q$. These are often well-behaved enough that inner parts of the multiplication data $(\sigma,\delta)$ can \emph{always} be untwisted from $\O[[x;\sigma,\delta]]$, making a study of these skew power series rings tractable using our methods.

Rings of this form are highly abundant, since beginning with a sufficiently nice filtered ring $R$, it is possible to produce such a ring $\O$ which is closely related to $R$ due to \cite[\S 3, Theorem C and proof]{ardakovInv}. For instance, the authors of the present paper proved the results of \cite{jones-woods-1} by relating skew power series rings over $R$ to skew power series rings over $Q(\O)$.

In short, we will take $Q$ to be a semisimple artinian ring throughout, and we will assume that it is complete with respect to a filtration $v_Q$. Naturally, by the Artin-Wedderburn theorem, $Q$ is isomorphic to a finite direct product of full matrix rings over division rings $F_1,\cdots,F_d$, and we will usually construct our maximal order $\O$ in $Q$ by simply taking maximal orders in the $F_i$. More specifically, we will assume that the data $(Q,\O,v_Q)$ satisfies some or all of the following hypotheses:

\textbf{Hypotheses.}

\begin{enumerate}[label=(H\arabic*),noitemsep]
\item We can realise $Q$ as a product $Q=A_1\times\cdots\times A_d$, where the rings $A_1,\dots,A_d$ form the minimal non-zero ideals of $Q$, $\mathcal{O}=\O_1\times\cdots\times\O_d$ for some maximal order $\O_i$ in $A_i$. and for each $i = 1, \dots, d$, we are given
\begin{itemize}
\item a complete \emph{discrete valuation ring} $D_i$ (defined as in \S \ref{subsec: DVRs} below),
\item its Goldie ring of quotients $F_i$ (with its induced filtration $v_{F_i}$: see \S \ref{subsec: DVRs}),
\item the full matrix rings $M_n(D_i) \subseteq M_n(F_i)$ (with the matrix filtration $M_n(v_{F_i})$: see Definition \ref{defn: filtrations}.2),
\item filtered ring isomorphisms $\iota_i: A_i \to M_n(F_i)$ such that $\O_i = \iota_i^{-1}(M_n(D_i))$.
\end{itemize}
In this context, we will write $D$ and $F$ for the products of the $D_i$ and $F_i$ respectively, and they will be given their respective product filtrations (see Definition \ref{defn: filtrations}.1), which we will sometimes denote $v_D$ and $v_F$.

It will also sometimes be convenient to identify $M_n(F) = M_n(F_1) \times \dots \times M_n(F_d)$. We will write $\iota: Q \to M_n(F)$ for the induced filtered isomorphism, and we assume that $v_Q=M_n(v_F)\circ\iota$.
\item The skew derivation $(\sigma,\delta)$ is \emph{compatible} with $v_Q$ (defined as in \S \ref{subsec: defining skew power series rings} below).
\item The automorphism $\sigma$ permutes the minimal nonzero ideals $A_1, \dots, A_d$ of $Q$ transitively.
\end{enumerate}

It will additionally be convenient to name the following hypothesis:

\begin{enumerate}
\item[(S)] In the context of (H1), $d = 1$: that is, $Q$ is \emph{simple} artinian, $F$ is a division ring, $D$ is a complete discrete valuation ring, etc.
\end{enumerate}

\textit{Remarks.}
\begin{itemize}[noitemsep]
\item Hypotheses (H1) + (S) are the context of \cite[\S 3, particularly 3.14]{ardakovInv}: our $Q$ is there called $Q(B)$, and it can be realised as a simple quotient of the artinian ring called $\widehat{Q}$.
\item It follows from (H1) that $v_F$ is the $J(D)$-adic filtration on $F$, and hence $v_Q$ is the $J(\mathcal{O})$-adic filtration on $Q$.
\item Hypothesis (H2) is a crucial hypothesis when working with filtered skew power series rings: many natural and important examples of filtered skew power series rings satisfy some kind of compatibility criterion (see e.g. \cite{jones-woods-1,letzter-noeth-skew,venjakob}, \cite[\S\S 2.4--2.5]{woods-SPS-dim}), and this compatibility criterion ensures that the ring multiplication is well defined (see \S\ref{subsec: reparametrising skew polynomials}). Note that, in particular, together with (H1) it implies that $\sigma$ preserves $\O$.
\item Hypothesis (H3) is a mild simplification. In fact, our results can also deal with the more general case where $Q \cong \prod_{i=1}^d M_{n_i}(F_i)$ (note that the $n_i$ may be different!) and $\sigma$ has multiple orbits, by applying the techniques of \S \ref{subsec: reducing to orbits} to reduce easily to a case satisfying (H3).
\end{itemize}

Assuming these hypotheses, our main result allows us to realise the skew power series ring $\O[[x;\sigma,\delta]]$ in a form that allows us to reduce to the study of skew-power series rings over the division rings $D_i$, a much less daunting task:

\begingroup
\setcounter{tmp}{\value{thm}}% store current value of theorem counter
\setcounter{thm}{0} %assign desired value to theorem counter
\renewcommand\thethm{\Alph{thm}}% locally redefine the representation of the theorem counter

\begin{thm}\label{A}
If $(Q,\O,v_Q)$ satisfy hypotheses (H1-3), with $F,D,v_F,\iota$ defined as in the statements of the hypotheses, then there exists a skew derivation $(\tau,\theta)$ on $F$, compatible with $v_F$, and an isomorphism of filtered rings $\varphi: \O[[x;\sigma,\delta]] \to M_n(D[[y;\tau,\theta]])$ extending $\iota|_{\O}$, where $y$ is the image of $ax-t$ for some $a\in \O^\times$, $t\in J(\O)$.
\end{thm}

\endgroup

Note that this statement makes sense, because if $(\tau,\theta)$ is compatible with $v_F$, which is the $J(D)$-adic filtration, then it follows that $\tau$ and $\theta$ preserve $D$, and hence $(\tau,\theta)$ restricts to a compatible skew-derivation of $D$.

It also follows from Theorem A that the Krull dimension of $\O[[x;\sigma,\delta]]$ is equal to the Krull dimension of $D[[y;\tau,\theta]]$, which is $2$, by similar methods to those of \cite[\S 3.1 and Theorem 3.3]{woods-SPS-dim}.

\subsection{Untwisting skew derivations}\label{subsec: reparametrising skew polynomials}

In order to prove our main result, we first find general conditions under which inner parts of the multiplication data $(\sigma,\delta)$ may be untwisted from $R[[x;\sigma,\delta]]$:

\textbf{Notation.} Given a ring $R$ and an invertible element $a\in R^\times$, we will write $\c_a$ for the inner automorphism of $R$ defined by $\c_a(r) = ara^{-1}$ for all $r\in R$. Also, given an element $t\in R$, we will write $\d_{\sigma,t}$ for the \emph{inner $\sigma$-derivation} of $R$ defined by $\d_{\sigma,t}(r) = tr - \sigma(r)t$ for all $r\in R$.

Let $R$ be a ring and $(\sigma,\delta)$ a skew derivation on $R$, and fix $a\in R^\times$ and $t\in R$. In the study of skew polynomial rings, it is often useful to \emph{reparametrise} the ring $R[x;\sigma,\delta]$, i.e. replace the variable $x$ with a new, more convenient variable $y\in R[x;\sigma,\delta]$, usually taken to be $y = ax$ or $y = x - t$. It is easy to see that $R[ax] = R[x-t] = R[x]$ as $R$-modules, and a calculation of the multiplication data shows that
$$(ax)r = a(\sigma(r)x + \delta(r)) = (a\sigma(r)a^{-1})(ax) + a\delta(r)$$
and
$$(x-t)r = \sigma(r)x + \delta(r) - tr = \sigma(r)(x-t) + \delta(r) - (tr - \sigma(r)t),$$
from which we can conclude that
\begin{itemize}
\item $R[x;\sigma,\delta] = R[ax; \c_a \sigma, a\delta]$,
\item $R[x;\sigma,\delta] = R[x-t; \sigma, \delta - \d_{\sigma,t}]$
\end{itemize}
as rings. This reparametrisation is a powerful tool in the study of skew polynomial rings, as it effectively implies that \emph{inner} automorphisms and $\sigma$-derivations can be ``untwisted" to become \emph{trivial}.

In the case of filtered skew power series rings $R[[x;\sigma,\delta]]$, we can no longer reparametrise arbitrarily due to the topology: that is, given a prospective new variable $y\in R[[x;\sigma,\delta]]$, it is no longer clear when $R[[x]] = R[[y]]$ as modules. Our second main result gives clear and broadly applicable sufficient conditions.

\begingroup
\setcounter{tmp}{\value{thm}}% store current value of theorem counter
\setcounter{thm}{1} %assign desired value to theorem counter
\renewcommand\thethm{\Alph{thm}}% locally redefine the representation of the theorem counter

\begin{thm}\label{B}
Let $(R,v)$ be a complete filtered ring, and suppose that $(\sigma,\delta)$ is a compatible skew derivation on $R$. Fix $a\in R^\times$ and $t\in R$.

\begin{enumerate}[label=(\roman*)]
\item If $v(a) = v(a^{-1}) = 0$, then $R^b[[x; \sigma, \delta]] = R^b[[ax; \c_a\sigma, \c_a\delta]]$ as filtered rings.
\item If $v(t) \geq 1$, then $R^b[[x; \sigma, \delta]] = R^b[[x-t; \sigma, \delta + \d_{\sigma,t}]]$ as filtered rings.
\end{enumerate}
\end{thm}

\endgroup

(By the phrase ``as filtered rings" here, we mean that they are equal as rings, and that the standard filtrations as defined in (\ref{eqn: x-filtration on SPS ring}) are equal: i.e. in the notation of (\ref{eqn: x-filtration on SPS ring}), we have $f_{v,x} = f_{v,ax}$ in part (i) and $f_{v,x} = f_{v,x-t}$ in part (ii).)

This is proved at the end of \S \ref{sec: reparametrising}.

\subsection{Ideal contraction and simplicity}

In $\S 5$ we will prove some results that follow as consequences from our main theorems. The first of these addresses the following question: when is $R[[x;\sigma,\delta]]$ a simple ring?

Let $R \subseteq S$ be rings. Then we will say that an ideal $I\lhd S$ is \emph{$R$-disjoint} if $I\cap R = 0$.

Let $R$ be a simple ring and $(\sigma,\delta)$ a skew derivation on $R$. It is often useful to ask when $R[x;\sigma,\delta]$ is also a simple ring: see e.g. \cite[\S 3]{CozFai08} or \cite{oinert-richter-silvestrov}. This is clearly equivalent to the statement that $R[x;\sigma,\delta]$ has no nonzero $R$-disjoint ideals. However, the ideal generated by $x$ is a nonzero $R$-disjoint ideal in the case when $\delta = 0$, and similarly -- by untwisting -- there exist nonzero $R$-disjoint ideals more generally when $\delta$ is inner. This suggests that inner derivations have a role to play in the simplicity of $R[x;\sigma,\delta]$.

The correct generalisation of ``inner" is as follows. The following are equivalent \cite[Theorem 2.6, Corollary 2.7]{LerMat92}:
\begin{enumerate}[label=(\roman*),noitemsep]
\item $R[x;\sigma,\delta]$ has nonzero $R$-disjoint ideals,
\item $\delta$ is a \emph{quasi-algebraic} $\sigma$-derivation, i.e. there exists an endomorphism $\theta$ of $R$, an inner $\theta$-derivation $D$ of $R$, and elements $0 \neq a_n, a_{n-1}, \dots, a_1, b\in R$ (for some $n \geq 1$) such that
$$a_n \delta^n(r) + a_{n-1} \delta^{n-1}(r) + \dots + a_1\delta(r) = bD(r)$$
for all $r\in R$. (In fact, $n$ and $\theta$ can be chosen so that $\theta = \sigma^n$ \cite[\S 2]{LerMat92}.)
\end{enumerate}
There are also equivalent conditions phrased in the language of \emph{invariant} and \emph{semi-invariant polynomials}. Many further such results, and references to the historical literature on these matters, are given in \cite{LamLerLeu89,LerMat92}. See \cite[Theorem 3.4]{irving2} or \cite[\S3]{CisFerGon90} for examples of the usefulness of conditions involving $R$-disjoint ideals.

\begingroup
\setcounter{tmp}{\value{thm}}% store current value of theorem counter
\setcounter{thm}{2} %assign desired value to theorem counter
\renewcommand\thethm{\Alph{thm}}% locally redefine the representation of the theorem counter

\begin{thm}\label{C}
If we assume that the data $(Q,\O,v_Q)$ satisfies Hypotheses (H1--3) + (S), then $\O[[x;\sigma,\delta]]$ has no nonzero $\O[x;\sigma,\delta]$-disjoint ideals. It follows that if $Q[x;\sigma,\delta]$ is a simple ring, then $Q\otimes_{\O} \O[[x;\sigma,\delta]]$ is a simple ring.
\end{thm}

\endgroup

\subsection{Uniform dimension}

Recall that the \emph{uniform dimension} (also called \emph{Goldie dimension} or \emph{Goldie rank}) of a right $R$-module $M$ is defined as follows. We set $\udim(M_R) = n$ if and only if there are uniform submodules $U_1, \dots, U_n \leq M$, pairwise intersecting in zero, such that $U_1 \oplus \dots \oplus U_n \leq M$ is an essential submodule \cite[2.2.9]{MR}. If $R$ is a ring, we write $\rudim(R) := \udim(R_R)$ for its (right) uniform dimension.

Uniform dimension is preserved under skew polynomial extensions in many cases of interest. For instance, Goodearl and Letzter showed that $\rudim(R[x;\sigma,\delta]) = \rudim(R)$ if $R$ is a prime noetherian ring \cite[Lemma 1.2]{GooLet94}, and Matczuk \cite{Mat95} showed that this equality holds in an even broader range of cases, including the case where $R$ is semiprime right Goldie.

More recently, the study of uniform dimension under skew power series extensions was initiated by Letzter and Wang in the paper \cite{letzter-wang-goldie}. Let $S = R[[x;\sigma]]$ be a \emph{pure automorphic} skew power series extension (i.e. $\delta = 0$): then, if $R$ is semiprime right noetherian, we have that $\rudim(S) = \rudim(R)$ by \cite[Theorem 2.8]{letzter-wang-goldie}. 

Of course, if $(R,v)$ is a complete positively filtered ring, $(\sigma,\delta)$ is a compatible skew derivation on $R$, and $\delta$ happens to be an \emph{inner} $\sigma$-derivation of the form described in Theorem A(ii), say $\delta = \d_{\sigma,t}$ for some $t\in R$ satisfying $v(t)\geq 1$, then it follows from Theorem A that $R[[x;\sigma,\delta]] = R[[y; \sigma]]$ after setting $y = x+t$. This puts us immediately into the context of \cite{letzter-wang-goldie}, allowing us to conclude that $\rudim(R[[x; \sigma,\delta]]) = \rudim(R)$ in this context too.

Now assume Hypotheses (H1--3) + (S) and their notation: in particular, recall that $\O \cong M_n(D)$.

We prove the following result only under these rather stringent restrictions, but this is (to our knowledge) the first such result for skew power series extensions with nontrivial derivations, and unlike the previous paragraph, it covers the case of some \emph{outer} $\sigma$-derivations using methods unlike those of \cite{letzter-wang-goldie}. We hope that, combined with the localisation process for filtered rings outlined in \cite[\S 3 and Theorem C]{ardakovInv}, this will spark further research for more general filtered skew power series rings. In \S \ref{subsec: udim}, we prove:

\begingroup
\setcounter{tmp}{\value{thm}}% store current value of theorem counter
\setcounter{thm}{3} %assign desired value to theorem counter
\renewcommand\thethm{\Alph{thm}}% locally redefine the representation of the theorem counter

\begin{thm}\label{D}
If we assume that the data $(Q,\O,v_Q)$ satisfies Hypotheses (H1-3) + (S), then $\rudim(\O[[x;\sigma,\delta]]) = \rudim(\O)$.
\end{thm}

\endgroup

\section{Preliminaries}

\subsection{Filtered rings and discrete valuation rings}

Our conventions for filtrations (which, in this paper, are always separated $\mathbb{Z}$-filtrations) are as follows.

A \emph{(ring) filtration} on a ring $R$ is a function $f: R\to \mathbb{Z}\cup\{\infty\}$ satisfying the following properties for all $r,s\in R$:
\begin{enumerate}[label=(\roman*),noitemsep]
\item $f(1) = 0$,
\item $f(r+s) \geq \min\{f(r),f(s)\}$,
\item $f(rs) \geq f(r) + f(s)$,
\item $f(r) = \infty$ if and only if $r = 0$.
\end{enumerate}
We will say that $(R,f)$ is a \emph{filtered ring} for short. If $f$ takes values in $\mathbb{N}\cup\{\infty\}$, we will say that $(R,f)$ is \emph{$\mathbb{N}$-filtered} or \emph{positively filtered}.

\begin{defn}\label{defn: filtrations}
$ $

\begin{enumerate}
\item If $(A, f_A)$ and $(B, f_B)$ are filtered rings, the \emph{product filtration} $f := f_A\times f_B$ on $R = A\times B$ is given by $f(a,b) = \min\{f_A(a), f_B(b)\}$.
\item If $(A, f)$ is a filtered ring and $n\geq 2$ is an integer, the \emph{matrix filtration} $g := M_n(f)$ on $M_n(A)$ is given by $g(\sum a_{ij} e_{ij}) = \min_{i,j} \{f(a_{ij})\}$, where $\{e_{ij}\}_{1\leq i,j\leq n}$ is the standard set of matrix units of $M_n(A)$.
\end{enumerate}
\end{defn}

\subsection{Skew derivations on semisimple artinian rings}

Let $R$ be a ring. The pair $(\sigma,\delta)$ is called a \emph{skew derivation} on $R$ if $\sigma\in\Aut(R)$ and $\delta$ is a \emph{(left) $\sigma$-derivation} of $R$, which means that $\delta$ is a linear map satisfying $\delta(rs) = \delta(r)s + \sigma(r)\delta(s)$ for all $r,s\in R$.

Here are some basic properties:

Suppose that $Q$ is a semisimple artinian ring (\emph{without} topology), say $Q = \prod_{i=1}^d A_i$ as a product of two-sided ideals, where each $A_i \cong M_{n_i}(F_i)$ as rings for some positive integers $n_i$ and division rings $F_i$. Suppose that $(\sigma, \delta)$ is a skew derivation on $Q$. We list some well-known facts.

\begin{props}\label{props: cauchon-robson facts about sigma and delta}
$ $

\begin{enumerate}[label=\arabic*.]
\item There exists a permutation $\rho$ of the indices $\{1, \dots, d\}$ such that $\sigma(A_i) = A_{\rho(i)}$ and $\delta(A_i) \subseteq A_i + A_{\rho(i)}$. Hence, if $\mathcal{S}$ is an orbit of $\rho$, then setting $B := \prod_{i\in \mathcal{S}} A_i$ and $\sigma' = \sigma|_{B}$, $\delta' = \delta|_{B}$, we get that $(\sigma', \delta')$ is a skew derivation of $B$. \cite[1.1--1.3]{cauchon-robson}
\item Suppose that $\rho$ permutes $\{1, \dots, d\}$ transitively. Then $n_1 = \dots = n_d$ ($= n$, say), so that there exists an isomorphism $\iota: Q \to M_n(F_1\times \dots \times F_d)$. Writing $F := F_1 \times \dots \times F_d$, we can then write $\sigma$ as $\eta\circ M_n(\tau)^\iota$, where $\eta$ is an inner automorphism of $Q$ and $\tau$ is an automorphism of $F$. \cite[2.1--2.4]{cauchon-robson}  (Here, and elsewhere, $M_n(\tau)^\iota$ means $\iota^{-1} M_n(\tau) \iota$.)
\item Suppose further that $\eta$ is trivial, so $\sigma = M_n(\tau)^\iota$. Then $\delta = \varepsilon + M_n(\theta)^\iota$, where $\varepsilon$ is an inner $\sigma$-derivation of $Q$ and $\theta$ is a $\tau$-derivation of $F$. \cite[2.5]{cauchon-robson}
\end{enumerate}
\end{props}

\begin{rk}
In fact, in the context of Property \ref{props: cauchon-robson facts about sigma and delta}.3, if $d > 1$ then something stronger holds: $\theta$ can be taken to be the zero map, so that $\delta$ itself is an inner $\sigma$-derivation \cite[1.4]{cauchon-robson}. However, in the context of \emph{filtered} rings, we will allow $\theta$ to be nonzero, as this extra flexibility is crucial for ensuring that the decomposition $\delta = \varepsilon + M_n(\theta)^\iota$ behaves well with respect to the filtration.
\end{rk}

\subsection{Compatible filtrations and skew power series rings}\label{subsec: defining skew power series rings}

\begin{defn}\label{defn: compatible}
Let $(R,v)$ be a filtered ring and $(\sigma,\delta)$ a skew derivation on $R$. We will say that $(\sigma,\delta)$ is \emph{(weakly) compatible} with $v$ if $v(\sigma(r)) = v(r)$ and $v(\delta(r)) > v(r)$ for all $0\neq r\in R$.
\end{defn}

\begin{rk}
This is more general than the notion of ``compatibility" used by the authors in \cite{jones-woods-1}, which could be called \emph{strong} compatibility.
\end{rk}

\begin{defn}
Let $(R,v)$ be a complete, positively filtered ring, i.e. $R$ is a ring admitting a separated discrete filtration $v: R\to \mathbb{N}\cup\{\infty\}$ with respect to which $R$ is complete.

The set $R[[x]] := \displaystyle \prod_{n\geq 0} Rx^n$, whose elements are formal sums $r_0 + r_1x + r_2 x^2 + \dots$ over arbitrary $r_i\in R$, is a left $R$-module. This is a complete filtered $R$-module with standard filtration 
$$f := f_{v,x}: R[[x]] \to \mathbb{N}\cup\{\infty\},$$ given by
\begin{equation}\label{eqn: x-filtration on SPS ring}
f\left(\sum_{i\geq 0} r_i x^i\right) = \inf_{i\geq 0} \{ v(r_i) + i\},
\end{equation}
which is separated and discrete. Note that $R[x]$ is dense in $R[[x]]$.

A skew derivation $(\sigma,\delta)$ on $R$ makes $R[x]$ into a ring, with multiplication determined uniquely by the rule $xr = \sigma(r)x + \delta(r)$. We write this ring as $R[x; \sigma, \delta]$.

If $(\sigma,\delta)$ is compatible with $v$, then it induces a well-defined associative multiplication on $R[[x]]$ in the same way: see \cite[Proposition 1.17]{jones-woods-1} (cf. \cite[Lemma 2.1]{venjakob} or \cite[\S 3.4]{letzter-noeth-skew}), and we denote this ring by $R[[x;\sigma,\delta]]$. The function $f$ on $R[[x;\sigma,\delta]]$ as defined above is a positive ring filtration, and $R[[x;\sigma,\delta]]$ is complete with respect to $f$.
\end{defn}

\subsection{Discrete valuation rings}\label{subsec: DVRs}

\begin{defn}\label{defn: DVRs}
Following Ardakov \cite{ardakovInv}, we will say that a \emph{discrete valuation ring} is a noetherian domain $D$ with the property that, for every nonzero $x\in Q(D)$ (the division ring of quotients), we have either $x\in D$ or $x^{-1}\in D$.
\end{defn}

We begin by showing that $D$ has properties very similar to those of commutative discrete valuation rings.

\begin{lem}\label{lem: right ideals in DVRs}
Let $D$ be a discrete valuation ring.
\begin{enumerate}[label=(\roman*),noitemsep]
\item $D$ is a local ring.
\item All right (resp. left) ideals of $D$ are principal.
\item The lattice of right (resp. left) ideals of $D$ is totally ordered.
\item All right (resp. left) ideals of $D$ are two-sided.
\end{enumerate}
\end{lem}

\begin{proof}
In the language of \cite{marubayashi-van-oystaeyen}, $D$ is a noetherian \emph{total subring} of the skew field $Q(D)$, and statements (i--iii) follow from \cite[Proposition 1.2.15]{marubayashi-van-oystaeyen}.

The proof of (iv) below is adapted from \cite[Lemma 1]{Bru69}. We give the proof for right ideals; the proof for left ideals is of course similar.

Suppose there exist right ideals of $D$ that are not two-sided, and let $J$ be the maximal such right ideal. By (ii), $J = aD$ for some $a\in D$: then, for some $r\in D$, we have $ra =: b \not\in aD$ by assumption. By (iii), this implies $aD \subsetneq bD$, and so $a = bs$ for some $s\in D$. Combining these two equations, we can see that $b = rbs$.

Now, by the maximality of $J$, we have that $Db \subseteq DbD = bD$, so that $rb = bt$ for some $t\in D$. In particular, $b = bts$, and so $b(1-ts) = 0$. But $b$ cannot be zero, so as $D$ is a domain, we must have $ts = 1$, and hence (as noetherian rings are Dedekind-finite) $st = 1$. It follows that $at = b$, contradicting the assumption that $b\not\in aD$.
\end{proof}

\begin{propn} \label{propn: ideals in DVRs}
Let $D$ be a discrete valuation ring.
\begin{enumerate}[label=(\roman*),noitemsep]
\item $J(D) = \pi D$ for some normal element $\pi$.
\item Every nonzero ideal of $D$ has the form $\pi^n D$ for some $n\in\mathbb{N}$.
\end{enumerate}
\end{propn}

\begin{proof}
(i) is an immediate consequence of Lemma \ref{lem: right ideals in DVRs}.

To show (ii): note that Lemma \ref{lem: right ideals in DVRs}(iv) also implies that $D$ is an FBN ring \cite[6.4.7]{MR}, and so $\bigcap_{n=1}^\infty \pi^n D = 0$ by \cite[Theorem 9.13]{GW}. We now argue exactly as in the commutative case: indeed, a nonzero ideal $aD$ must satisfy $\pi^{n+1} D\subsetneq aD\subseteq \pi^{n}D$ for some $n$ by Lemma \ref{lem: right ideals in DVRs}(iii), from which it follows that $a = \pi^n u$ for some $u\in D\setminus \pi D$, which must be a unit by Lemma \ref{lem: right ideals in DVRs}(i).
\end{proof}

An element $\pi$ as in the above proposition will be called a \emph{uniformiser} of $D$.

In the following, $D$ will continue to denote a complete discrete valuation ring, and we will also set $F = Q(D)$, $\O = M_n(D)$ and $Q = M_n(F)$. Also write $v_F$ for the induced $J(D)$-adic filtration on $F$, and suppose that $(\tau, \theta)$ is a skew derivation on $F$ compatible with $v_F$; likewise write $v_Q$ for the $J(\O)$-adic filtration on $Q$, and suppose that $(\sigma, \delta)$ is a skew derivation on $Q$ compatible with $v_Q$. This puts us essentially in the situation of Hypotheses (H1--3) + (S). The following is now routine to check.

\begin{cor}\label{cor: localising SPS rings at pi}
$ $

\begin{enumerate}[label=(\roman*), noitemsep]
\item Let $\mathcal{S}$ be the multiplicatively closed set in $D$ generated by $\pi$. Then $F = \mathcal{S}^{-1}D = D\mathcal{S}^{-1}$. Moreover, $\pi$ is normal in $D[[y;\tau,\theta]]$, and $F\otimes_D D[[y;\tau,\theta]] = \mathcal{S}^{-1}D[[y;\tau,\theta]] = D[[y;\tau,\theta]]\mathcal{S}^{-1}$.
\item Let $\mathcal{S}$ be the multiplicatively closed set in $\O$ generated by $\pi$ (where we identify $D$ with its diagonal embedding in $\O$). Then $Q = \mathcal{S}^{-1}\O = \O\mathcal{S}^{-1}$. Moreover, $\pi$ is normal in $\O[[x;\sigma,\delta]]$, and $Q\otimes_\O \O[[x;\sigma,\delta]] = \mathcal{S}^{-1}\O[[x;\sigma,\delta]] = \O[[x;\sigma,\delta]]\mathcal{S}^{-1}$.\qed
\end{enumerate}
\end{cor}

\section{Reparametrising filtered skew derivations}\label{sec: reparametrising}

Throughout this section, let $(R,v)$ be a complete, positively filtered ring. Suppose also that $R$ admits a skew derivation $(\sigma,\delta)$ which is compatible with $v$, and take $y\in R[x]$ such that $R[x] = R[y]$ (an equality of left $R$-modules).

Given an element $\sum_{i=0}^m r_i y^i\in R[y]$, we may define the function
$$f_{v,y} : \sum_{i=0}^m r_i y^i \mapsto \inf_{i\geq 0} \{v(r_i) + i\}.$$
Note that $f_{v,x}$ is the standard ring filtration defined in (\ref{eqn: x-filtration on SPS ring}). In contrast, for arbitrary elements $y\in R[x]$, the function $f_{v,y}$ will not always be a ring filtration, and even when it is, it will generally not be equivalent to $f_{v,x}$.

\begin{ex}\label{ex: when reparametrisation doesn't work}
Take $y := x + 1 \in \mathbb{Z}_p[x]$, and $v$ the $p$-adic valuation on $\mathbb{Z}_p$. Then for all $n$, we have $f_{v,x}((x+1)^n) = 0$ but $f_{v,y}((x+1)^n) = n$. In particular, $\mathbb{Z}_p[x] = \mathbb{Z}_p[y]$ but $\mathbb{Z}_p[[x]] \neq \mathbb{Z}_p[[y]]$.
\end{ex}

In this section, we identify two families of elements $y\in R[x]$ for which $f_{v,x}$ and $f_{v,y}$ are equal as functions.

\subsection{Conditions for identical filtrations}

With notation as above, we first show that $f_{v,x}$ and $f_{v,y}$ are equal when $y = x - t$ for some $t\in R$ satisfying $v(t) \geq 1$.

Write $(x-t)^n = x^n + \beta_{n,1}x^{n-1} + \dots + \beta_{n,n-1} x + \beta_{n,n}$ for all $n$.

\begin{lem}\label{lem: values of structure constants, additive case}
For all $n, i$, we have $v(\beta_{n,i}) \geq i$.
\end{lem}

\begin{proof}
Firstly, note that $(x-t)\beta_{n,i} x^{n-i} = \sigma(\beta_{n,i}) x^{n+1-i} + (\delta(\beta_{n,i}) - t\beta_{n,i})x^{n-i}$, so (writing $\beta_{n,0} := 1$ for ease of notation) we may calculate $(x-t)^{n+1}$ as
\begin{align*}
(x-t)\left( \sum_{i=0}^n \beta_{n,i} x^{n-i}\right) &= \sum_{i=0}^n \sigma(\beta_{n,i}) x^{n+1-i} + \sum_{j=0}^n (\delta(\beta_{n,j}) - t(\beta_{n,j})) x^{n-j}\\
&= x^{n+1} + \sum_{i=1}^n (\sigma(\beta_{n,i}) + \delta(\beta_{n,i-1}) - t\beta_{n,i-1}) x^{n+1-i} + (\delta(\beta_{n,n}) - t\beta_{n,n}),
\end{align*}
by setting $j = i + 1$ in the second sum. That is,
$$\begin{cases}
\beta_{n+1,0} = 1,\\
\beta_{n+1,i} = \sigma(\beta_{n,i}) + \delta(\beta_{n,i-1}) - t\beta_{n,i-1}& (1\leq i\leq n),\\
\beta_{n+1,n+1} = \delta(\beta_{n,n}) - t\beta_{n,n},
\end{cases}$$
from which the claim follows by induction on $n$.
\end{proof}

\begin{lem}\label{lem: subtracting t from the variable}
Let $p(x)\in R[x]$ be a polynomial, and $t\in R$ such that $v(t) \geq 1$. Then $f_{v,x}(p(x-t)) \geq f_{v,x}(p(x))$.
\end{lem}

\begin{proof}
Write $p(x) = r_0 + r_1x + \dots + r_m x^m$. Then
\begin{align*}
p(x-t) = \sum_{n=0}^m r_n (x-t)^n &= \sum_{n=0}^m r_n\left( \sum_{i=0}^n  \beta_{n,i} x^{n-i}\right)& \text{setting } \beta_{n,0} := 1\\
&= \sum_{j= 0}^m \left(\sum_{n=j}^m r_n \beta_{n,n-j} \right)x^j & \text{where } j := n - i,
\end{align*}
and so
\begin{align*}
f_{v,x}(p(x-t)) &= f_{v,x}\left(\sum_{j=0}^m \left(\sum_{n=j}^m r_n \beta_{n,n-j} \right)x^j\right)\\
&= \inf_{0\leq j\leq m} \left\{ v\left(\sum_{n=j}^m r_n \beta_{n,n-j}\right) + j\right\}\\
&\geq \inf_{0\leq j\leq m} \inf_{j\leq n\leq m} \left\{ v(r_n) + v(\beta_{n,n-j}) + j\right\}& \text{as } v \text{ is a filtration} \\
&\geq \inf_{0\leq j\leq m} \inf_{j\leq n\leq m} \left\{ v(r_n) + n\right\} & \text{by Lemma \ref{lem: values of structure constants, additive case}}\\
&= \inf_{0\leq n\leq m} \left\{ v(r_{n}) + n\right\}\\
&= f_{v,x} \left( \sum_{n=0}^m r_n x^n\right) = f_{v,x}(p(x)).
\end{align*}
\end{proof}

\begin{propn}\label{propn: additive reparametrisation}
Let $t\in R$ such that $v(t) \geq 1$. Then $f_{v,x} = f_{v,x-t}$.
\end{propn}

\begin{proof}
Take an arbitrary element $p(x) \in R[x]$. Then

\begin{align*}
f_{v,x}(p(x)) &\leq f_{v,x}(p(x+t))& \text{applying Lemma \ref{lem: subtracting t from the variable} to } -t\\
&= f_{v,x-t}(p(x)) &\text{changing variables } x \mapsto x-t \text{ throughout}\\
&\leq f_{v,x-t}(p(x-t)) & \text{applying Lemma \ref{lem: subtracting t from the variable} to } t\\
&= f_{v,x}(p(x))& \text{changing variables } x \mapsto x+t \text{ throughout},
\end{align*}
from which we can conclude that $f_{v,x}(p(x)) = f_{v,x-t}(p(x))$.
\end{proof}

Next, we show that $f_{v,x}$ and $f_{v,y}$ are equal when $y = ax$ where $a\in R^\times$ and $v(a) = v(a^{-1}) = 0$.

Write $(ax)^n = \gamma_{n,0} x^n + \gamma_{n,1} x^{n-1} + \dots + \gamma_{n,n-1} x + \gamma_{n,n}$.

\begin{lem}\label{lem: values of structure constants, multiplicative case}
$v(\gamma_{n,i}) \geq i$.
\end{lem}

\begin{proof}
As in Lemma \ref{lem: values of structure constants, additive case}, calculating $(ax)^{n+1} = ax(\gamma_{n,0} x^n + \gamma_{n,1} x^{n-1} + \dots + \gamma_{n,n-1} x + \gamma_{n,n})$ gives
$$\begin{cases}
\gamma_{n+1,0} = a\sigma(\gamma_{n,0}),\\
\gamma_{n+1,i} = a\sigma(\gamma_{n,i}) + a\delta(\gamma_{n,i-1})& (1\leq i\leq n),\\
\gamma_{n+1,n+1} = a\delta(\gamma_{n,n}),
\end{cases}$$
and we may perform induction as in Lemma \ref{lem: values of structure constants, additive case}.
\end{proof}

\begin{lem}\label{lem: multiplying the variable by a}
$f_{v,x}(p(ax)) \geq f_{v,x}(p(x))$.
\end{lem}

\begin{proof}
\begin{align*}
p(ax) = \sum_{n=0}^m r_n (ax)^n &= \sum_{n=0}^m r_n\left(\sum_{i=0}^n \gamma_{n,i} x^{n-i}\right)\\
&= \sum_{j=0}^m \left(\sum_{n=j}^m r_{n} \gamma_{n,n-j}\right) x^j & \text{where } j := n - i,
\end{align*}
and so
\begin{align*}
f_{v,x}(p(ax)) &= f_{v,x}\left( \sum_{j=0}^m \left(\sum_{n=j}^m r_{n} \gamma_{n,n-j}\right) x^j\right),
\end{align*}
and the proof now proceeds exactly as in the proof of Lemma \ref{lem: subtracting t from the variable}.
\end{proof}

\begin{propn}\label{propn: multiplicative reparametrisation}
$f_{v,x} = f_{v,ax}$.
\end{propn}

\begin{proof}
Take an arbitrary element $p(x) \in R[x]$. Then
\begin{align*}
f_{v,x}(p(x)) &\leq f_{v,x}(p(a^{-1}x))& \text{applying Lemma \ref{lem: multiplying the variable by a} to } a^{-1}\\
&= f_{v,ax}(p(x)) &\text{changing variables } x \mapsto ax \text{ throughout}\\
&\leq f_{v,ax}(p(ax)) & \text{applying Lemma \ref{lem: multiplying the variable by a} to } a\\
&= f_{v,x}(p(x))& \text{changing variables } x \mapsto a^{-1}x \text{ throughout},
\end{align*}
from which we can conclude that $f_{v,x}(p(x)) = f_{v,ax}(p(x))$.
\end{proof}

Finally, we fix any $y\in R[x]$ such that $R[x] = R[y]$ and $f_{v,x} = f_{v,y}$. Denote this common filtration by $f$. It follows that:

\begin{thm}\label{thm: reparametrisation conditions}
$R[[x]] = R[[y]]$ as filtered left $R$-modules.
\end{thm}

\textit{Proof of Theorem \ref{B}.} In case (i), set $y = ax$, so that $f := f_{v,x} = f_{v,y}$ by Proposition \ref{propn: multiplicative reparametrisation}; in case (ii), set $y = x-t$, and use Proposition \ref{propn: additive reparametrisation}. In both cases, $R[x] = R[y]$. Now Theorem \ref{thm: reparametrisation conditions} implies that $R[[y]]$ and $R[[x]]$ can be identified as filtered modules, and the multiplication data has already been calculated in \S \ref{subsec: reparametrising skew polynomials}, so the conclusion follows. \qed

\section{Skew derivations on semisimple artinian rings}

\subsection{Reducing to orbits}\label{subsec: reducing to orbits}

For now, we do \textit{not} assume any of the Hypotheses (H1--3), and we let $(R,v)$ be an arbitrary filtered ring such that $R$ admits a decomposition $R \cong B\times C$ (as unfiltered rings).

We will abuse notation and write $R = B\times C$ (as unfiltered rings). We will also write $B$ (resp. $C$) for the ideal $B\times 0$ (resp. $0\times C$) of $R$, so that there are inclusion maps $j_B: B\to R$ and $j_C: C\to R$ and projection maps $\pi_B: R\to B$ and $\pi_C: R\to C$.

Suppose further that the filtration on $R$ is complete and positive, and that $R$ admits a skew derivation $(\sigma,\delta)$ which restricts to skew derivations on $B$ and $C$. (That is, setting $\sigma_B = \pi_B \sigma j_B$ and $\delta_B = \pi_B \delta j_B$, we have that $(\sigma_B, \delta_B)$ is a skew derivation on $B$, and likewise for $C$.)

Write $v_B$ and $v_C$ for the restrictions of $v$ to $B$ and $C$ respectively. In general, even if $(\sigma,\delta)$ is compatible with $v$, it may not be true that $(\sigma_B, \delta_B)$ is compatible with $v_B$, and so we must restrict to the case in which the decomposition $R\cong B\times C$ and the filtration $v$ interact nicely. The following lemma is an immediate consequence of the definition of the product filtration, as in Definition \ref{defn: filtrations}.

\begin{lem}\label{lem: skew derivations restrict to sigma-orbits}
If $v = v_B \times v_C$, then $(\sigma_B, \delta_B)$ is compatible with $v_B$, and $(\sigma_C, \delta_C)$ is compatible with $v_C$.\qed
\end{lem}

Hence, under the assumption $v = v_B \times v_C$, we may define the filtered rings
\begin{itemize}
\item $B[[x_B; \sigma_B, \delta_B]]$, with filtration $f_B$, satisfying $f_B|_B = v_B$ and $f_B(x_B) = 1$,
\item $C[[x_C; \sigma_C, \delta_C]]$, with filtration $f_C$, satisfying $f_C|_C = v_C$ and $f_C(x_C) = 1$,
\end{itemize}
as in (\ref{eqn: x-filtration on SPS ring}).

\begin{propn}\label{propn: reducing to sigma-orbits}
Suppose that $v = v_B\times v_C$. Then there is an isomorphism of filtered rings 
$\varphi: R[[x; \sigma, \delta]] \to B[[x_B; \sigma_B, \delta_B]] \times C[[x_C; \sigma_C, \delta_C]]$.
\end{propn}

\begin{proof}
It is straightforward to check that the maps
\begin{align*}
\varphi: R[[x; \sigma, \delta]] &\to B[[x_B; \sigma_B, \delta_B]] \times C[[x_C; \sigma_C, \delta_C]]\\
\sum_{i\geq 0} r_i x^i &\mapsto \left(\sum_{i\geq 0} \pi_B(r_i) x_B^i, \sum_{i\geq 0} \pi_C(r_i) x_C^i\right)
\end{align*}
and
\begin{align*}
\theta: B[[x_B; \sigma_B, \delta_B]] \times C[[x_C; \sigma_C, \delta_C]] &\to R[[x; \sigma, \delta]]\\
\left(\sum_{i\geq 0} b_i x_B^i, \sum_{i\geq 0} c_i x_C^i\right)&\mapsto \sum_{i\geq 0} (j_B(b_i),j_C(c_i)) x^i
\end{align*}
are the mutually inverse filtered isomorphisms as required.
\end{proof}

Let $Q'$ be an arbitrary semisimple artinian filtered ring admitting a skew derivation $(\sigma,\delta)$, and write the minimal nonzero ideals of $Q'$ as $A_1,\dots, A_e$, so that $Q' = \prod_{i=1}^e A_i$. In the case where the $A_i$ fall into several $\sigma$-orbits, write $\rho$ for the permutation of the indexing set $\{1, \dots, e\}$ induced on the set $\{A_1, \dots, A_e\}$ by $\sigma$ as in Property \ref{props: cauchon-robson facts about sigma and delta}.1. Let $\mathcal{S}$ be a union of (some) orbits of $\rho$ and $\mathcal{S}' = \{1, \dots, e\} \setminus \mathcal{S}$, and suppose (to avoid trivial cases) that both $\mathcal{S}$ and $\mathcal{S}'$ are nonempty. Now set $B' = \prod_{i\in\mathcal{S}} A_i$ and $C' = \prod_{i\in \mathcal{S}'} A_i$: it follows from Property \ref{props: cauchon-robson facts about sigma and delta}.1 that $(\sigma_{B'}, \delta_{B'})$ restricts to a skew derivation on $B'$, and likewise for $C'$.

Assume now that each $A_i \cong M_{n_i}(F_i)$, where $n_i \geq 1$ is some positive integer and $F_i$ is the Goldie ring of quotients of a complete discrete valuation ring $D_i$. Set $\O_i$ to be the preimage in $A_i$ of $M_{n_i}(D_i)$, so that $\O' = \O_1 \times \dots \times \O_e$ is a maximal order in $Q'$. Moreover, if $j=\rho(i)$ then $A_j=\sigma(A_i)$ so $M_{n_i}(F_i)\cong M_{n_j}(F_j)$, which implies that $n_i=n_j$ and $F_i\cong F_j$.  

Suppose that all of these rings are given their natural filtrations: that is,
\begin{itemize}[noitemsep]
\item each $D_i$ retains its discrete valuation, and $F_i$ inherits the $J(D_i)$-adic valuation (see \S \ref{subsec: DVRs}),
\item $M_{n_i}(D_i)$ and $M_{n_i}(F_i)$ are given the corresponding matrix filtrations (see Definition \ref{defn: filtrations}.2),
\item $\O_i$ and $A_i$ inherit their filtrations from $M_{n_i}(D_i)$ and $M_{n_i}(F_i)$ under the above isomorphisms, and
\item $\O'$ and $Q'$ are given the product filtrations (see Definition \ref{defn: filtrations}.1) from the $\O_i$ and $A_i$ respectively.
\end{itemize}
Then $\O'$, $B'$ and $C'$ as defined above will satisfy the hypotheses of Proposition \ref{propn: reducing to sigma-orbits}. Moreover, if $\mathcal{S}$ is taken to be a single orbit of $\rho$ with $|\mathcal{S}| = d$, then (after renumbering so that $\mathcal{S} = \{1, \dots, d\}$ and writing $n = n_1 = \dots = n_d$) the ring $\O := \O'\cap B'$ as defined above, its Goldie ring of quotients $Q := B'$, etc. will satisfy Hypotheses (H1--3).

\subsection{Untwisting inner automorphisms}

In this subsection, we assume the full force of Hypotheses (H1--3) and adopt their notation.

Without loss of generality, reordering the $A_i$ if necessary, we will set $\sigma(A_i) = A_{i+1}$ for $1\leq i\leq d-1$ and $\sigma(A_d) = A_1$.

We may now invoke Property \ref{props: cauchon-robson facts about sigma and delta}.2. In particular, there is a decomposition $\sigma = \eta \circ M_n(\tau)^\iota$, where $\eta$ is an inner automorphism of $Q$, say $\eta = \c_a$ for some $a\in Q^\times$, and $\tau$ is an automorphism of $F$.

\begin{lem}\label{lem: eta and M_n(tau) preserve O}
Both $\eta$ and $M_n(\tau)^\iota$ preserve $\O$.
\end{lem}

\begin{proof}
Since $\eta$ is an inner automorphism of $Q$, it will preserve each $A_i$. But $\sigma(A_i) = A_{i+1}$ (with indices interpreted modulo $d$), so $M_n(\tau)^\iota$ must send $A_i$ to $A_{i+1}$, and hence $\tau$ sends $F_i$ to $F_{i+1}$. However, since $\sigma$ and $\eta$ are continuous, it follows that $\tau$ is continuous, and so $D_{i+1} = \O_{i+1} \cap F_{i+1} = \tau(D_i)$, hence $M_n(\tau)^\iota$ preserves $\O$. Now it follows that $\eta = \sigma \circ M_n(\tau^{-1})^{\iota}$ preserves $\O$.
\end{proof}

Our aim in this subsection is to ``untwist" $\eta$ by making a change of variables $x \mapsto x'$, i.e. find an element $x'\in \O[[x; \sigma, \delta]]$ such that
$$\O[[x; \sigma, \delta]] = \O[[x'; M_n(\tau)^\iota, \delta']].$$
By Theorem \ref{B}(i), it would suffice if $v(a) = v(a^{-1}) = 0$: this would imply that $a, a^{-1}\in \O$, and we could then set $x' = a^{-1}x$, giving $\delta' = a^{-1}\delta$ as in \S \ref{subsec: reparametrising skew polynomials}.

Of course, in general, $a$ will \emph{not} necessarily have this property: for instance, if $\O$ is a complete discrete valuation ring with central uniformiser $\pi$, then $\c_a = \c_{\pi^r a}$ for all $r\in\mathbb{Z}$, and $v(\pi^r a)$ will usually not be zero. Surprisingly, this naive obstruction is the only kind of obstruction that occurs.

\begin{propn}\label{propn: reducing to case sigma = M_n(tau)}
There exist an element $b\in Q^\times$, an inner automorphism $\eta' = \c_b$ of $Q$, and an automorphism $\tau'$ of $F$ such that $\sigma = \eta' \circ M_n(\tau')^\iota$ and $v(b) = v(b^{-1}) = 0$.
\end{propn}

\begin{proof}
Suppose $a = (a_1, \dots, a_d)\in Q^\times$: then, by Lemma \ref{lem: eta and M_n(tau) preserve O}, we have $a_i\O_i a_i^{-1} = \O_i$ for each $i$. Let $k_i = v(a_i)$, i.e. $a_i\in J(\O_i)^{k_i} \setminus J(\O_i)^{k_i + 1}$. So, if $\pi_j$ is a uniformiser of $D_j$, $I_j$ is the identity matrix of $M_n(D_j)$, and $\tilde{\pi}_j := \iota^{-1}(\pi_j I_j)$, then we have $a_j = b_j \tilde{\pi}_j^{k_j}$ for some $b_j\in \O_j$ with $v(b_j) = 0$. Set $b = (b_1, \dots, b_d)$, so that $v(b) = 0$.

By Corollary \ref{cor: localising SPS rings at pi}(ii), $\tilde{\pi}_j$ is normal in $\O_j$, so the right ideal $b_j\O_j$ is in fact a two-sided ideal. Moreover, by Proposition \ref{propn: ideals in DVRs}(ii) and Morita equivalence, as the ideal $b_j\O_j$ contains the element $b_j$ of value $0$, it must be equal to $\O_j$. Hence $b_j$ is a \emph{unit} in $\O_j$, and we have $v(b_j) = v(b_j^{-1}) = 0$.

Now set $\eta'(r) = brb^{-1}$ for all $r\in Q$ and $\tau'(s) = \Pi \tau(s) \Pi^{-1}$, where $\Pi = (\pi_1^{k_1} I_1, \dots, \pi_d^{k_d} I_d)$. The claim now follows from a short calculation.
\end{proof}

\subsection{Untwisting inner $\sigma$-derivations}\label{subsec: untwisting inner derivations}

In this subsection, let $(R,f)$ be an arbitrary $\mathbb{Z}$-filtered ring admitting a compatible skew derivation $(\sigma,\delta)$, and write the $f$-level sets of $R$ as $F_n R$ for $n\in \mathbb{Z}$. We will also suppose that
\begin{itemize}
\item $R = M_n(A)$ for some ring $A$,
\item $f = M_n(g)$ for some filtration $g$ on $A$,
\item $\sigma = M_n(\tau)$ for some automorphism $\tau$ of $A$, and
\item $\delta = M_n(\theta) + \varepsilon$, where $\theta$ is a $\tau$-derivation of $A$ and $\varepsilon$ is an \emph{inner} $\sigma$-derivation of $R$, say $\varepsilon = \d_{\sigma,u}$ for some $u\in R$.
\end{itemize}
(Compare Property \ref{props: cauchon-robson facts about sigma and delta}.3.) We will write the standard set of matrix units in $R$ as $\{e_{ij}\}_{1\leq i,j\leq n}$.

Our aim in this subsection is to ``untwist" $\varepsilon$ by making a change of variables $x \mapsto x'$, i.e. find an element $x'\in R[[x; M_n(\tau), \delta]]$ such that
$$R_{\geq 0}[[x; M_n(\tau), \delta]] = R_{\geq 0}[[x'; M_n(\tau), M_n(\theta)]]$$
where of course, $R_{\geq 0}$ is the positively filtered subring of $R$. As before, we would be done by Theorem \ref{B}(ii) if we had $f(u) \geq 1$. This is also unreasonable to expect, albeit this time for slightly less obvious reasons: for instance, if $A$ is the division ring of fractions of a complete discrete valuation ring with uniformiser $\pi$, then $M_n(\theta) + \d_{M_n(\tau),u} = M_n(\theta + \d_{\tau,\pi^r}) + \d_{M_n(\tau),u-\pi^r I}$ for all $r\in \mathbb{Z}$, and $v(u - \pi^r I)$ can be less than $1$. However, again, under mild conditions this is the only obstruction that occurs.

\begin{propn}\label{propn: can untwist inner derivations given a matrix filtration}
In the above setup, there exist an element $u'\in R$, a $\tau$-derivation $\theta'$ of $A$, and an inner $\sigma$-derivation $\varepsilon' = \inn(u')$ of $R$ such that $\delta = M_n(\theta') + \varepsilon'$ and $f(u') \geq 1$.
\end{propn}

\begin{proof}
Write $u = \sum_{i,j} u_{ij} e_{ij}$ for some coefficients $u_{ij}\in A$. Let $1\leq p,q\leq n$ be arbitrary, and consider the matrix unit $e_{pq}\in R$. By assumption, $\delta(e_{pq})\in F_1R$, and so
\begin{align}\label{eqn: epsilon of matrix units}
\varepsilon(e_{pq}) \equiv -M_n(\theta)(e_{pq}) \mod F_1R.
\end{align}
But we can calculate the left-hand side of this congruence (\ref{eqn: epsilon of matrix units}) explicitly as
$$\varepsilon(e_{pq}) = \sum_{i,j} (u_{ij}e_{ij}e_{pq} - e_{pq}u_{ij}e_{ij}) = \sum_i u_{ip}e_{iq} - \sum_j u_{qj}e_{pj},$$
and the right-hand side of (\ref{eqn: epsilon of matrix units}) is just $-\theta(1_A)e_{pq}$, which is zero. So we may rewrite (\ref{eqn: epsilon of matrix units}) as $\sum_i u_{ip}e_{iq} - \sum_j u_{qj}e_{pj} \equiv 0 \mod F_1R$, and equate corresponding entries, to get
\begin{align*}
\begin{cases}
u_{ip} \in F_1A & i \neq p\\
u_{qj} \in F_1A & j \neq q\\
u_{pp} - u_{qq} \in F_1A,
\end{cases}
\end{align*}
and so, as $p$ and $q$ were arbitrary, we get $u \equiv u_{11}1_R \mod F_1R$.

Now setting $u' := u - u_{11}1_R$, and defining $\theta'(a) := \theta(a) + u_{11}a - \tau(a)u_{11}$ and $\varepsilon' := \d_{M_n(\tau),u'}$, we are done.
\end{proof}

Upshot: in this case, using Theorem \ref{B}(i) we can pass to the case when $\delta = M_n(\theta)$.

\textit{Proof of Theorem \ref{A}.} Firstly, Proposition \ref{propn: reducing to case sigma = M_n(tau)} shows that there exist some $\tau\in\Aut(F)$ and some $b\in Q^\times$ satisfying $v(b) = v(b^{-1}) = 0$ such that $\sigma = \c_b \circ M_n(\tau)^\iota$, and both of these preserve $\O$ by the same argument as in Lemma \ref{lem: eta and M_n(tau) preserve O}. So by Theorem \ref{B}(i), we may set $x' = b^{-1}x$ to get $\O[[x;\sigma,\delta]] = \O[[x'; M_n(\tau)^\iota, \delta']]$ for the $M_n(\tau)^\iota$-derivation $\delta' := b^{-1} \delta$.

Secondly, Proposition \ref{propn: can untwist inner derivations given a matrix filtration} shows that there exist some $\tau$-derivation $\theta$ of $F$ and some $u\in Q$ satisfying $v(u) \geq 1$ such that $\delta' = M_n(\theta)^\iota + \d_{M_n(\tau)^\iota,u}$. So by Theorem \ref{B}(ii), we may set $x'' = x' - u$ to get $\O[[x'; M_n(\tau)^\iota, \delta']] = \O[[x''; M_n(\tau)^\iota, M_n(\theta)^\iota]]$.

Finally, the maps
\begin{align*}
\O[[x''; M_n(\tau)^\iota, M_n(\theta)^\iota]]&\to M_n(D)[[y; M_n(\tau), M_n(\theta)]]\\
\sum_{i\geq 0} q_i (x'')^i&\mapsto \sum_{i\geq 0} \iota(q_i) y^i
\end{align*}
and
\begin{align*}
M_n(D)[[y; M_n(\tau), M_n(\theta)]] &\to M_n(D[[z; \tau, \theta]])\\
\sum_{i\geq 0} \left(\sum_{j,k=1}^n c_{ijk}e_{jk}\right) y^i &\mapsto \sum_{j,k=1}^n \left(\sum_{i\geq 0} c_{ijk} z^i\right) e_{jk}
\end{align*}
can now be checked to be filtered ring isomorphisms, and $v_\O = M_n(v_D)\circ \iota$. \qed%

\section{Applications}\label{sec: applications}

Throughout this section, we assume our data $(Q,\O,v_Q)$ satisfies Hypotheses (H1--3) + (S).

\subsection{Polynomial elements}\label{subsec: polynomial elements}

\begin{defn}
A \emph{polynomial element} of $R[[x]]$ (resp. $R[[x;\sigma,\delta]]$) is an element of $R[x]$ (resp. $R[x;\sigma,\delta]$).
\end{defn}

We first recall the following important generalisation of the Weierstrass preparation theorem for skew power series rings, essentially due to Venjakob.

\begin{thm}\label{thm: right ideals of the SPS ring over D contain polynomials}
Every nonzero right ideal of $D[[y; \tau, \theta]]$ contains a nonzero polynomial element.
\end{thm}

\begin{proof}
Let $v$ be the $J(D)$-adic filtration and $\pi$ a uniformiser for $D$. Then every nonzero element $r\in D[[y; \tau, \theta]]$ can be written as $r = (s_0 + s_1 y + s_2 y^2 + \dots)\pi^m$, where all $s_i\in D$ and $\inf_{i\geq 0}\{v(s_i)\} = 0$, by Corollary \ref{cor: localising SPS rings at pi}(i). Hence $s = s_0 + s_1 y + s_2 y^2 + \dots$ satisfies the hypotheses of \cite[Theorem 3.1]{venjakob}, and so a right-hand version of \cite[Corollary 3.2]{venjakob} tells us that $s$ can be expressed uniquely as $s = Pu$, where $u\in D[[y; \tau, \theta]]^\times$ and $P\in D[y; \tau, \theta]$.

In particular, if $A$ is a nonzero right ideal of $D[[y;\tau,\theta]]$, then given any nonzero $r\in A$, we can write it as $r = Pu\pi^m$ as above. Since $\pi$ is normal in $D[[y;\tau,\theta]]$ by Corollary \ref{cor: localising SPS rings at pi}(i), this is just $r = P\pi^m u'$ for some unit $u'\in D[[y;\tau,\theta]]$, and hence the polynomial element $P \pi^m$ is also an element of $A$.
\end{proof}

\textit{Remark.} Corollary \ref{cor: localising SPS rings at pi}(i) also implies a similar result for $F\otimes_D D[[y;\tau,\theta]]$.

\begin{cor}\label{cor: ideals of the SPS ring over O contain polynomials}
Every nonzero (two-sided) ideal of $\O[[x;\sigma,\delta]]$ contains a nonzero polynomial element.
\end{cor}

\begin{proof}
Theorem \ref{A} tells us that there exists an isomorphism $\varphi: \O[[x;\sigma,\delta]] \to M_n(D[[y; \tau, \theta]])$ extending $\iota$, such that $y = \varphi(ax-t)$. In particular, if $P\in D[[y; \tau, \theta]]$ is polynomial in the variable $y$, then $\varphi^{-1}(PI)$ (where $I$ is the identity matrix) is polynomial in the variable $x$. The result now follows from Theorem \ref{thm: right ideals of the SPS ring over D contain polynomials} and by Morita equivalence.
\end{proof}

\textit{Proof of Theorem \ref{C}.} The first statement is simply restating Corollary \ref{cor: ideals of the SPS ring over O contain polynomials}.

For the second statement, take a nonzero ideal $I$ of $Q\otimes_\O \O[[x;\sigma,\delta]]$. Then $J := I\cap \O[[x; \sigma, \delta]]$ is clearly an ideal of $\O[[x; \sigma, \delta]]$. By Corollary \ref{cor: localising SPS rings at pi}(ii), multiplying any nonzero element of $I$ by an appropriate power of the regular element $\pi$ will give an element of $J$, so $J \neq 0$.

Hence, by the first statement, we know that $J\cap \O[x;\sigma,\delta] \neq 0$, and so $I\cap Q[x;\sigma,\delta] \neq 0$. But as we are assuming that $Q[x;\sigma,\delta]$ is simple, it must follow that $I\cap Q[x;\sigma,\delta] = Q[x;\sigma,\delta]$. In particular, $1\in I$, and so $I = Q\otimes_\O \O[[x;\sigma,\delta]]$. \qed

\subsection{Uniform dimension}\label{subsec: udim}

This subsection continues the work of \cite{letzter-wang-goldie} on uniform dimensions (Goldie ranks) of skew power series rings. As we have already remarked in the Introduction, Theorem \ref{B} sometimes allows us to reduce directly to the results of \cite{letzter-wang-goldie} when the derivation is inner. However, in the special case of Hypotheses (H1--3) + (S), we can now prove similar results about skew power series rings over $\O$ for \emph{arbitrary} derivations.

\textit{Proof of Theorem \ref{D}.} By Theorem \ref{A}, we know that $\O[[x;\sigma,\delta]] \cong M_n(D[[y; \tau, \theta]])$ for some appropriate skew derivation $(\tau, \theta)$, and hence that $\rudim(\O[[x;\sigma,\delta]]) = n(\rudim(D[[y; \tau, \theta]]))$ \cite[Example 2.11(iii)]{MR}. In the same way, $\rudim(\O) = n(\rudim(D))$, and of course $\rudim(D) = 1$ as $D$ is a noetherian integral domain \cite[Example 2.11(i)]{MR}.

It remains to show that $\rudim(D[[y;\tau,\theta]]) = 1$. So let $U$ be a uniform right ideal of $D[[y;\tau,\theta]]$, and let $I$ be an arbitrary nonzero right ideal. By Theorem \ref{thm: right ideals of the SPS ring over D contain polynomials}, $U\cap D[y;\tau,\theta] \neq 0$ and $I\cap D[y;\tau,\theta] \neq 0$, and so as $D[y;\tau,\theta]$ is a prime ring \cite[Theorem 1.2.9(iii)]{MR}, their intersection $(U\cap I)\cap D[y;\tau,\theta]$ is also nonzero. In particular, this shows that $U$ is an essential right ideal of $D[[y;\tau,\theta]]$, and so $\rudim(D[[y;\tau,\theta]]) = 1$. \qed

\bibliography{biblio}
\bibliographystyle{plain}

\end{document}